\documentclass[12pt,reqno]{amsart}
\usepackage[a4paper,margin=2.5cm]{geometry}

\usepackage{amsmath,amsthm,amssymb}
\usepackage{amssymb}
\usepackage{amsmath}
\usepackage[all]{xy}
\usepackage{tikz-cd}
\usepackage{tikz}
\usepackage{hyperref}
\usepackage{upgreek}
\usepackage{booktabs}
\usepackage{array}
\usepackage{longtable}
\usepackage{ragged2e}
\usepackage{multirow}
\usepackage{enumitem}
\usepackage[table]{xcolor}
\usepackage{dutchcal}

\renewcommand{\arraystretch}{1.15}
\theoremstyle{plain} 
\newtheorem{theorem}{Theorem}[section]

\newtheorem{lemma}[theorem]{Lemma}
\newtheorem{corollary}[theorem]{Corollary}

\numberwithin{equation}{section}

\theoremstyle{definition} \newtheorem{definition}[theorem]{Definition}
\newtheorem*{definition*}{Definition}
\newtheorem*{question*}{Question}

\theoremstyle{plain}

\newtheorem*{remark}{Remark}

\usepackage[all]{xy}

\newcommand{\ncom}{\newcommand}
\ncom{\mylabel}[1]{{\rm (#1)}\label{#1}}
\ncom{\Hom}{{\textit{Hom}}}
\ncom{\eop}{{\hfill $\Box$}}

\newcommand{\C}{{\mathbb C}}

\def\min{{\rm min}}
\def\card{{\rm card}}

\def\Aut{{\rm Aut}}

\def\Hom{{\rm Hom}}

\def\Kum{{\rm Kum}}

\def\Sp{{\rm Sp}}
\def\SO{{\rm SO}}
\def\SU{{\rm SU}}
\def\Prod{{\rm Prod}}
\def\F{{\rm \mathbb{F}}}

\begin{document}

\title[Classification of products of Hilbert scheme of points on surface]{On the classification of products of Hilbert schemes of points over a surface}

\author[A. Mukherjee]{Arijit Mukherjee}
\address{Department of Mathematics, IIT Madras, Chennai, India}
\email{mukherjee7.arijit@gmail.com}
\author[A. Pahari]{Anubhab Pahari}
\address{Department of Mathematics, IIT Madras, Chennai, India}
\email{anubhabpahari@gmail.com, ma22d012@smail.iitm.ac.in}

\begin{abstract}
Let $S$ be a smooth projective surface over $\mathbb{C}$ and $S^{[n]}$ be the Hilbert scheme of $n$ points over $S$, for any positive integer $n$. Under suitable hypotheses, we show that the products of Hilbert schemes $S^{[{\bf a}]}$ and $S^{[{\bf b}]}$ associated to two distinct partitions ${\bf a}$ and ${\bf b}$ of $n$  are non-isomorphic, using  Betti numbers, Hodge numbers and Euler characteristics of the individual factors.  The classifications obtained using Betti numbers and Euler characteristics extend over $\overline{\F}_q$ as well, where $\F_q$ is the finite field with $q$ elements.  We also provide a complete classification of such product spaces for K3 surfaces using their inherent symplectic structure.
\end{abstract}

\keywords{Hilbert scheme of points over a surface, irreducible symplectic manifold, Betti number, Hodge polynomial, coloured partition, majorization, K3 surface}

\subjclass[2020]{14C05, 14C30, 14Q10, 14Q15, 14J28,  05A17.}

\maketitle

\tableofcontents

\section{Introduction}\label{Section 1}
Let $S$ be a smooth projective surface over $\mathbb{C}$. We denote the Hilbert scheme of $n$ points over a surface $S$ by $S^{[n]}$. It is a smooth projective variety over $\mathbb{C}$ of dimension $2n$ and is a (crepant) resolution of the singularities of the $n^{\text{th}}$ symmetric product $S^{(n)}$ (cf. \cite{MR237496}).

A \textit{partition} of a positive integer $n$ is a tuple ${\bf a}=(n_1,\ldots,n_r)$ of positive integers such that $n_1\geq \cdots \geq n_r$ and $\sum_{i=1}^r n_i=n$. We define \textit{the Hilbert scheme of points over $S$ associated to ${\bf a}$}, denoted by $S^{[{\bf a}]}$, as:
\begin{equation*}
    S^{[{\bf a}]}:= S^{[n_1]}\times \cdots \times S^{[n_r]}.
\end{equation*} 
In this paper, we focus on the following:
\begin{question*}\label{main question}
 Let ${\bf a}$ and ${\bf b}$ be two distinct partitions of a positive integer $n$. Is it true that the $2n$-dimensional schemes $S^{[{\bf a}]}$ and $S^{[{\bf b}]}$ are non-isomorphic ?   
\end{question*}

The above question is completely resolved for the curve case (cf. \cite{MR4833710}). 

 In Section \ref{Section: classification for product of K3s}, we give a complete answer when $S$ is a K3 surface.  We use the \textit{Beauville decomposition} theorem to obtain the following result (cf. Theorem \ref{Classification of product of hilb scheme of K3} and Theorem \ref{aut gp of prod of hilb for k3}).
\begin{theorem}\label{Zeroth main theorem_Introduction}
    Let $S$ be a $K3$ surface.  Then, 
    \begin{enumerate}
        \item $S^{[{\bf a}]}$ is isomorphic to $ S^{[{\bf b}]}$ if and only if ${\bf a}={\bf b}$.
        \item $\operatorname{Aut}(S^{[\bf a]}) \cong \prod_{i=1}^r \Big({\operatorname{Aut}(S^{[n_i]})}^{\nu_{{\bf a}}(n_i)} \rtimes \mathfrak{S}_{\nu_{{\bf a}}(n_i)}\Big),$
        where $\nu_{{\bf a}}(n_i)$ is the multiplicity of $n_i$ in ${\bf a}=(n_1,\ldots,n_r)$.
    \end{enumerate}
\end{theorem}

An analogous result holds for products of generalised Kummer varieties as well (cf. Remark on p.~4-5).

In Section \ref{Section: classification of product for any surface}, we deal with the classification problem for any smooth projective surface $S$.  An effective approach for such problems is to use invariants of $S^{[n]}$ that behave well under products.  We use the generating functions for  Betti numbers and Hodge numbers (cf. \cite{thesis}, \cite{MR1612677}, \cite{MR1032930} and \cite{MR1312161}) to obtain the following result (cf. Theorem \ref{thm: main theorem for diff lengths}, Theorem \ref{thm: same length with b0(S) ge 1}, Theorem \ref{thm: main theorem for same length for b_0(S)=1 and b_1(S)=0} and Theorem \ref{thm:classification_same length_condition involves first betti number and minimal part}):
\begin{theorem}\label{Introduction_main theorems together}
    Let ${\bf a}=(n_1,\dots,n_r)$ and ${\bf b}=(m_1,\ldots,m_s)$ be two distinct partitions of a positive integer $n$ of length $r$ and $s$ respectively. Then $S^{[{\bf a}]}$ is not isomorphic to $S^{[{\bf b}]}$ in each of the following cases:
\begin{enumerate}
    \item $b_0(S)=1$ and one of the following holds:
    \begin{enumerate}
        \item $b_1(S)=0$, $b_2(S)\ne \tfrac{\nu_{{\bf b}}(1)-\nu_{{\bf a}}(1)}{s-r}-1$ and $r< s$.
        \item $b_1(S)=0$, $r=s$ and $\nu_{{\bf a}}(1)\ne \nu_{{\bf b}}(1)$.
        \item $b_1(S)>0$, $r<s$. 
        \item $b_1(S)\ge 2(\min\{n_j,m_j\}+1)$, $j$ being the largest integer such that $n_{j}\ne m_{j}$ and $r=s$.
        \end{enumerate}
    \item $b_0(S)>1$ and one of the following holds:
    \begin{enumerate}
        \item $r<s$ and ${\bf b}$ is a refinement of ${\bf a}$ (cf. Definition \ref{def: refinement}).
        \item $r=s$ and ${\bf b}$ strictly majorizes ${\bf a}$ (cf. Definition  \ref{Definition_majorization}).
    \end{enumerate}
\end{enumerate}
\end{theorem}

Next, we observe that the Euler characteristic $\chi(S^{[n]})$ is the same as the number of partitions of $n$ obtained using $\chi(S)$ many colours (cf. \eqref{28}). This identification allows us to apply the results of \cite{bringmann}, from which we obtain the following (cf. Theorem \ref{Classification_same length case_using majorization}):
\begin{theorem}\label{Fourth main theorem_Introduction}
 Let ${\bf a}$ and ${\bf b}$ be two partitions of a positive integer $n$, of the same length, such that ${\bf b}$ strictly majorizes ${\bf a}$ (cf. Definition  \ref{Definition_majorization}). Then 
 $S^{[{\bf a}]}$ is not isomorphic to $S^{[{\bf b}]}$, when $\chi(S)\geq 3$.
\end{theorem}

Although we present our results for surfaces over $\mathbb{C}$, most of those hold for any smooth projective surface over $\overline{\F}_q$ as well, where $\F_q$ is the finite field with $q$ elements (cf. Remark \ref{Remark_validity over alg closed fields over finite fields}).

The following table gives examples of surfaces for which Theorem \ref{Introduction_main theorems together} and Theorem \ref{Fourth main theorem_Introduction} hold.

\begin{longtable}{c p{5cm} cccc}
\caption{Betti numbers and Euler characteristics of smooth projective surfaces}
\label{Table}\\

\toprule
Kodaira \\ dimension & Surface type & $b_0(S)$ & $b_1(S)$ & $b_2(S)$ & $\chi(S)$ \\
\midrule
\endfirsthead

\multicolumn{6}{c}{\tablename\ \thetable\ -- continued from previous page} \\
\toprule
Kodaira \\ dimension & Surface type & $b_0(S)$ & $b_1(S)$ & $b_2(S)$ & $\chi(S)$ \\
\midrule
\endhead

\midrule
\multicolumn{6}{r}{Continued on next page} \\
\endfoot

\bottomrule
\endlastfoot

\multirow{4}{*}{$-\infty$}
& Del Pezzo surface of degree $d$, $1\!\le\!d\!\le\!9$
  & 1 & 0 & $10\!-\!d$ & $12\!-\!d$ \\[2pt]

& Hirzebruch surface $F_n=\mathbb{P}_{\mathbb{P}^1}(\mathcal{O}\!\oplus\!\mathcal{O}(n))$, $n\ge 1$
  & 1 & 0 & 2 & 4 \\[2pt]

& Rational elliptic surface $E(1)$
  & 1 & 0 & 10 & 12 \\[2pt]

& Ruled surface over a curve of genus $g$
  & 1 & $2g$ & 2 & $4(1\!-\!g)$ \\[2pt]

\midrule

\multirow{4}{*}{$0$}
& K3 surface & 1 & 0 & 22 & 24 \\
& Enriques surface & 1 & 0 & 10 & 12 \\
& Abelian surface & 1 & 4 & 6 & 0 \\
& Bielliptic surface & 1 & 2 & 2 & 0 \\

\midrule

\multirow{3}{*}{$1$}
& Elliptic surface over genus $g$ curve, $\chi(\mathcal O_S)=1$
  & 1 & $2g$ & $4g+10$ & 12 \\
& Elliptic surface over genus $g$ curve, $\chi(\mathcal O_S)=2$
  & 1 & $2g$ & $4g+22$ & 24 \\
& Elliptic surface $E(n)$, $n\ge 3$ (base $\mathbb{CP}^1$)
  & 1 & 0 & $12n\!-\!2$ & $12n$ \\[2pt]

\midrule

\multirow{2}{*}{$2$}
& $C_1 \times C_2$, $\text{genus}(C_i)=g_i>1$
  & 1 & $2(g_1+g_2)$ & $2+4g_1g_2$ & $4(1-g_1)(1-g_2)$ \\
& Quintic surface in $\mathbb{CP}^3$
  & 1 & 0 & 53 & 55 \\

\end{longtable}

It is clear from Theorem \ref{Introduction_main theorems together} and Theorem \ref{Fourth main theorem_Introduction} that our objective remains unsolved in a few cases.  For example, we cannot conclude whether $S^{[{\bf{a}}]}$ and $S^{[{\bf{b}}]}$ are non-isomorphic when ${\bf a}$ and ${\bf b}$ are not comparable through refinement for different lengths or majorization for the same lengths, for surfaces $S$ with $b_0(S)>1$. Nevertheless, we expect an affirmative answer, as it has been verified computationally in several cases. In this regard, we provide the following combinatorial conjecture, which originally appeared in \cite[Section 4]{bringmann}. \\
\textbf{Conjecture} (\textit{combinatorial version}): Let ${\bf a}$ and ${\bf b}$ be two distinct partitions of the same length of a positive integer $n$.  Then $p_k({\bf b}) \ne p_k({\bf a})$ for all $k\ge 4$.

A direct translation of this conjecture into the context of this article, via equation \eqref{28}, implies that for any two distinct partitions ${\bf a}$ and ${\bf b}$, of the same length, of a positive integer $n$, the $2n$-dimensional schemes $S^{[{\bf a}]}$ and $S^{[{\bf b}]}$ are not isomorphic if $\chi(S)\ge 4$.  However, following Theorem \ref{Introduction_main theorems together} and Theorem \ref{Fourth main theorem_Introduction},  we believe that a much stronger statement holds.  We formulate this as a conjecture.\\
\textbf{Conjecture} (\textit{geometric version}): Let ${\bf a}$ and ${\bf b}$ be two distinct partitions of a positive integer $n$. Then, $S^{[{\bf a}]}$ and $S^{[{\bf b}]}$ are not isomorphic.

%
\section{Classification of products of Hilbert schemes of points over a K3 surface}\label{Section: classification for product of K3s}  

It is well known that if $X$ is a simply connected compact K\"ahler manifold with vanishing Ricci curvature, then $X$ splits, uniquely up to K\"ahler isomorphism and reordering, as a product (cf. \cite[Theorem 1, p. 759]{Beauville}):
\begin{equation}\label{explicit form of the decomposition}
    \mathbb{C}^k \times \prod_i V_i \times \prod_j X_j\;,
\end{equation}
where, 
\begin{enumerate}[label=\alph*.]
\item $\mathbb{C}^k$ is equipped with standard K\"ahler metric.
\item Each $V_i$ is a simply connected compact K\"ahler manifold of real dimension $2m_i$ with holonomy group $\SU(m_i)\subset \SO(2m_i)$.
\item Each $X_j$ is a simply connected compact K\"ahler manifold of real dimension $4r_j$ with holonomy group $\Sp(r_j)\subset \SO(4r_j)$.
\end{enumerate}

In this section, by $S$ we mean a K3 surface. It is well known that $S^{[n]}$ is an irreducible symplectic manifold (cf. \cite[\S 2.2, p.~78]{Huybrechts}). Therefore, it has holonomy group $\Sp(n)$ (cf. \cite[Proposition 4, p.~763]{Beauville}) and $c_1(S^{[n]})=0$. Consequently, $S^{[\bf a]}$ is a simply connected compact K\"ahler manifold with $c_1(S^{[{\bf a}]})=0$.
\begin{theorem}\label{Classification of product of hilb scheme of K3}
    Let ${\bf a}=(n_1,\ldots,n_r)$ and ${\bf b}=(m_1,\ldots,m_s)$ be two  partitions of a positive integer $n$.  Then $S^{[{\bf a}]}$ is isomorphic to $ S^{[{\bf b}]}$ if and only if ${\bf a}={\bf b}$.
\end{theorem}
\begin{proof}
   The converse is trivial, so we only prove the forward direction. Recall that an isomorphism between projective varieties implies a biholomorphism in the complex analytic category. Hence, for a biholomorphism  $f: S^{[{\bf a}]} \rightarrow S^{[{\bf b}]}$, $(S^{[{\bf a}]}, f^{\ast}g_{S^{[{\bf b}]}},f^{\ast}\omega_{S^{[{\bf b}]}})$ and $(S^{[{\bf b}]},g_{S^{[{\bf b}]}},\omega_{S^{[{\bf b}]}})$ are isomorphic as K\"ahler manifolds, where $g_{S^{[{\bf b}]}}$ is the unique Ricci-flat K\"ahler metric (cf. \cite[Theorem 2, p. 364]{Yau}) associated with a K\"ahler form $\omega_{S^{[{\bf b}]}}$ on $S^{[{\bf b}]}$.  

   Therefore, by \eqref{explicit form of the decomposition} we obtain $r=s$ and $S^{[n_i]}\cong S^{[m_{\sigma(i)}]}$, where $\sigma \in \mathfrak{S}_r$, the symmetric group on $r$ symbols. This implies $n_i=m_{\sigma(i)}$ for all $i$ and hence ${\bf a} = {\bf b}$.
\end{proof}
 
Another consequence of the decomposition \eqref{explicit form of the decomposition} is that we can compute the automorphism group $\operatorname{Aut}(S^{[\bf a]})$. 
\begin{theorem}\label{aut gp of prod of hilb for k3} 
    Let ${\bf a}=(\underbrace{n_1,\ldots,n_1}_{\nu_{\bf a}(n_1)\; \text{times}},\ldots,\underbrace{n_r,\ldots,n_r}_{\nu_{\bf a}(n_r)\; \text{times}})$ be a partition of a positive integer $n$. Then $$\operatorname{Aut}(S^{[\bf a]}) \cong \prod_{i=1}^r \Big({\operatorname{Aut}(S^{[n_i]})}^{\nu_{{\bf a}}(n_i)} \rtimes \mathfrak{S}_{\nu_{{\bf a}}(n_i)}\Big),$$
     where $\nu_{{\bf a}}(n_i)$ is the multiplicity of $n_i$ in ${\bf a}$.
\end{theorem}
\begin{proof}
    By \eqref{explicit form of the decomposition}, there exists a surjective group homomorphism 
    $$\operatorname{Aut}(S^{[\bf a]}) \longrightarrow \prod_{i=1}^r \mathfrak{S}_{\nu_{{\bf a}}(n_i)},$$
     by sending automorphisms to the permutations induced by the repeating factors of $S^{[\bf a]}$. The kernel of this map is $\prod_{i=1}^r {\operatorname{Aut}(S^{[n_{i}]})}^{\nu_{{\bf a}}(n_i)}$. Therefore, we have a short exact sequence:
    \begin{equation*}
         1\longrightarrow \prod_{i=1}^r {\operatorname{Aut}(S^{[n_{i}]})}^{\nu_{{\bf a}}(n_i)} \longrightarrow \operatorname{Aut}(S^{[\bf a]}) \longrightarrow \prod_{i=1}^r \mathfrak{S}_{\nu_{{\bf a}}(n_i)}\longrightarrow 1.
     \end{equation*}
    This exact sequence splits via the natural inclusion of each $\mathfrak{S}_{\nu_{{\bf a}}(n_i)}$ permuting the factors. So, the assertion follows.
\end{proof}
\begin{remark}
 Similar results can be obtained for the products of generalised Kummer varieties. The generalised Kummer variety, denoted by $\text{Kum}^{n+1}(A)$ for a positive integer $n$, is an irreducible symplectic manifold of dimension $2n$ (cf. \cite[Proposition 7, p. 769]{Beauville}).  Given any $r$-tuple ${\bf a}=(n_1,\ldots,n_r)$ of positive integers, we define $\text{Kum}^{[\bf a]}(A)$ as follows:
\begin{equation*}\label{Kummer variety associated to partition}
 \text{Kum}^{[\bf a]}(A):= \text{Kum}^{n_1+1}(A) \times \cdots \times \text{Kum}^{n_r+1}(A).   
\end{equation*}
\begin{corollary}\label{Classification of product of generalised Kummer varieties}
Let ${\bf a}$ and ${\bf b}$ be two partitions of a positive integer $n$.  Then $\Kum^{[{\bf a}]}(A)$ is isomorphic to $\Kum^{[{\bf b}]}(A)$ if and only if ${\bf a}={\bf b}$.
\end{corollary}
\begin{proof}
    Similar to the proof of Theorem \ref{Classification of product of hilb scheme of K3}.
\end{proof}
  A similar result to Theorem \ref{aut gp of prod of hilb for k3} holds for $\Aut(\Kum^{[{\bf a}]}(A))$ as well.
\end{remark}
\section{Classification of products of Hilbert schemes of points over a smooth projective surface}\label{Section: classification of product for any surface}

Throughout this section, by $S$ we denote a smooth projective surface over $\C$ or $\overline{\F}_q$ unless otherwise stated. The section is organised into two subsections: the first treats partitions of different lengths, and the second treats partitions of the same length. In the latter, we perform a case-by-case analysis by fixing the values of $b_i(S)$  and $\chi(S)$.

 From \cite[Theorem 0.1, Equation 1(b), p.~193]{MR1032930}, we recall the generating function of the \textit{$i^{\text{th}}$ Betti number} $b_i(S^{[n]})$ of $S^{[n]}$:
\begin{equation}\label{Poincare series for HIlbert scheme of points on surface}
    \sum_{n=0}^{\infty} P(S^{[n]}, z) t^n = \prod_{m=1}^{\infty} \frac{(1 + z^{2m-1} t^m)^{b_1(S)} (1 + z^{2m+1} t^m)^{b_1(S)}}{(1 - z^{2m-2} t^m)^{b_0(S)} (1 - z^{2m} t^m)^{b_2(S)} (1 - z^{2m+2} t^m)^{b_0(S)}}\;, 
\end{equation}
where $P(S^{[n]}, z)=\sum_{i=0}^{2n}b_{i}(S^{[n]})z^{i}$.

From \eqref{Poincare series for HIlbert scheme of points on surface}, it is clear that 
\begin{equation}\label{kth betti number of nth Hilbert scheme is the coefficient of z^kt^n}
    b_{k}(S^{[n]})=\text{Coefficient\; of\; } z^{k}t^{n} \text{\;in\; the \;R.H.S\; of\;}\eqref{Poincare series for HIlbert scheme of points on surface}. 
\end{equation}

We note some lower Betti numbers of $S^{[n]}$. 
\begin{lemma}\label{relations between Betti numbers of S and its Hilbert scheme}
Let $n$ be any positive integer.  Then, 
\begin{enumerate}
    \item $ b_{0}(S^{[n]})=\begin{pmatrix}n+b_{0}(S)-1\\b_{0}(S)-1\end{pmatrix}$\;.
    \item $b_{1}(S^{[n]})=b_1(S)\begin{pmatrix}n+b_{0}(S)-2\\b_{0}(S)-1\end{pmatrix}$\;.
    \item If $b_0(S)=1$ and $b_1(S)=0$, then
\begin{equation*}
b_2(S^{[n]})= \left\{ \begin{array}{ll} b_2(S)+1 & \mbox{if $n>1$};\\ \\ b_2(S) &
\mbox{if $n=1$}.\end{array} \right.
\end{equation*}
\end{enumerate}
\end{lemma}
\begin{proof}
Follows from \eqref{kth betti number of nth Hilbert scheme is the coefficient of z^kt^n}.
\end{proof}
\subsection{Classification related to partitions of different lengths}
We introduce a partial order on partitions of different lengths of a positive integer. 
\begin{definition}\label{def: refinement}
Let ${\bf a}=(n_1,\dots,n_r)$ and ${\bf b}=(m_1,\dots,m_s)$ be two partitions of a positive integer $n$ with $r<s$.
\begin{enumerate}
\item[(i)] We say that ${\bf b}$ is a \emph{primitive refinement} of ${\bf a}$ if $s=r+1$ and there exist $i \in \{1,\dots,r\}$ and $j_1 \neq j_2 \in \{1,\dots,r+1\}$ such that
$n_i = m_{j_1} + m_{j_2}$,
and after removing $n_i$ from ${\bf a}$ and $m_{j_1}, m_{j_2}$ from ${\bf b}$, the remaining parts coincide as multisets.

\item[(ii)] We say that ${\bf b}$ is a \emph{refinement} of ${\bf a}$ if there exists a finite sequence of partitions
\[
{\bf a} = {\bf a}^{(0)}, {\bf a}^{(1)}, \dots, {\bf a}^{(t)} = {\bf b}
\]
such that ${\bf a}^{(k+1)}$ is a primitive refinement of ${\bf a}^{(k)}$.
\end{enumerate}
\end{definition}
\begin{lemma}\label{lem:refinement lemma}
    Let ${\bf a}=(n_1,\ldots,n_r)$ and ${\bf b}=(m_1,\ldots,m_s)$ be two partitions of a positive integer $n$ such that ${\bf b}$ is a refinement of ${\bf a}$. Then for any integer $p \ge 1$,
    \begin{enumerate}
        \item $\prod_{i=1}^{r}\dfrac{n_{i}+p}{p} <\prod_{j=1}^{s}\dfrac{m_{j}+p}{p}$ \;.
        \item $\prod_{i=1}^{r}\begin{pmatrix} n_{i}+p\\p \end{pmatrix} <\prod_{j=1}^{s}\begin{pmatrix} m_{j}+p\\p \end{pmatrix}$\;.
    \end{enumerate}
\end{lemma}
\begin{proof}
\begin{enumerate}
        \item For convenience, we denote:
    \begin{equation}\label{eqn: notation for product of (n+p)/p}
        \operatorname{Prod}({\bf a}):=\prod_{i=1}^{r}\dfrac{n_{i}+p}{p}.
    \end{equation}
It is enough to prove the inequality assuming that ${\bf b}$ is a primitive refinement of ${\bf a}$. Since,
    $$\Big(\frac{m_{j_1}+p}{p}\Big)\Big(\frac{m_{j_2}+p}{p}\Big)-\Big(\frac{n_i+p}{p} \Big)=\frac{m_{j_1}m_{j_2}}{p^2}>0,$$
    we have,
    \begin{equation*}
        \begin{split}
            \operatorname{Prod}({\bf a}) & =\prod_{i=1}^{r}\frac{n_{i}+p}{p} \\& = \Big(\prod_{i'\ne i}\frac{n_{i'}+p}{p}\Big)\Big(\frac{n_i+p}{p}\Big) \\ & < \Big(\prod_{j' \ne j_1,j_2}\frac{m_{j'}+p}{p}\Big)\Big(\frac{m_{j_1}+p}{p}\Big)\Big(\frac{m_{j_2}+p}{p}\Big) \\ & = \prod_{j=1}^{r+1}\frac{m_{j}+p}{p} = \operatorname{Prod}({\bf b}).
        \end{split}
    \end{equation*}
    \item  We prove statement $(2)$ by induction on $p$. The basis step $p=1$ follows from taking $p=1$ in statement (1). Assume that the claim is true for $p-1$. Hence, by induction hypothesis and statement (1), we have:
    \begin{equation*}
        \begin{split}
            &\prod_{i=1}^{r}\begin{pmatrix} n_{i}+p\\p \end{pmatrix}= \prod_{i=1}^{r}\begin{pmatrix} n_{i}+p-1\\p-1\end{pmatrix}\prod_{i=1}^{r}\dfrac{n_{i}+p}{p} \\ <&\prod_{j=1}^{s}\begin{pmatrix} m_{j}+p-1\\p-1\end{pmatrix}\prod_{j=1}^{s}\dfrac{m_{j}+p}{p}=\prod_{j=1}^{s}\begin{pmatrix} m_{j}+p\\p \end{pmatrix}.
        \end{split}
    \end{equation*}
 Therefore, the inductive step follows and hence the assertion.   
\end{enumerate}
\end{proof}
\begin{theorem}\label{thm: main theorem for diff lengths}
    Let ${\bf a}=(n_1,\ldots,n_r)$ and ${\bf b}=(m_1,\ldots,m_s)$ be two partitions of a positive integer $n$ with $r < s$. Then, $S^{[{\bf a}]}$ is not isomorphic to $S^{[{\bf b}]}$,
    provided one of the following conditions holds:
    \begin{enumerate}
        \item $b_0(S)>1$ and ${\bf b}$ is a refinement of ${\bf a}$.
        \item $b_0(S)=1$ and $b_1(S)>0$.
        \item $b_0(S)=1$, $b_1(S)=0$ and $$b_2(S)\neq \frac{\nu_{{\bf b}}(1)-\nu_{{\bf a}}(1)}{s-r}-1,$$ where $\nu_{{\bf a}}(1)$ and $\nu_{{\bf b}}(1)$ are  multiplicities of $1$ in ${\bf a}$ and ${\bf b}$ respectively.
    \end{enumerate}
\end{theorem}

\begin{proof} By K\"unneth formula, we have:
\begin{equation}\label{Poincare polynomial of product is product of Poincare polynomials}
\text{Poincar\'e\;polynomial\;of\;}S^{[{\bf a}]}=\sum_{j=0}^{4n} b_j(S^{[{\bf a}]})t^j=\prod_{i=1}^r\bigg(\sum_{j=0}^{4n_i} b_j(S^{[{n_i}]})t^j\bigg).
\end{equation}
\begin{enumerate}
    \item As $b_0(S)>1$, we have $b_{0}(S^{[n]})>1$ (cf. Lemma \ref{relations between Betti numbers of S and its Hilbert scheme} (1)). Also, comparing the constant terms from \eqref{Poincare polynomial of product is product of Poincare polynomials}, we have:
    \begin{equation*}
      b_{0}(S^{[{\bf a}]})= \prod_{i=1}^r b_{0}(S^{[n_{i}]}).  
    \end{equation*}
For $p= b_{0}(S)-1\geq 1$, by Lemma \ref{relations between Betti numbers of S and its Hilbert scheme}(1) and Lemma \ref{lem:refinement lemma} (2), we have: 
    \begin{equation*}
          \begin{split}
              b_{0}(S^{[{\bf b}]})-b_{0}(S^{[{\bf a}]}) & = \prod_{j=1}^{s}\begin{pmatrix} m_{j}+p\\p \end{pmatrix} -\prod_{i=1}^{r}\begin{pmatrix} n_{i}+p\\p \end{pmatrix}>0.
          \end{split}  
    \end{equation*} 
    
\item As $b_{0}(S)= 1$, by Lemma \ref{relations between Betti numbers of S and its Hilbert scheme} (1), we have $b_{0}(S^{[n]})=1$, $b_{1}(S^{[n]})=b_{1}(S)$.  Therefore, comparing the coefficients of $t$ from \eqref{Poincare polynomial of product is product of Poincare polynomials}, we have:
   \begin{equation*}
    b_{1}(S^{[{\bf a}]})= \sum_{i=1}^r b_{1}(S^{[n_{i}]}).   
   \end{equation*}
Therefore, as $b_{1}(S)> 0$ and $s>r$, we have:
\begin{equation*}
   b_{1}(S^{[{\bf a}]})=rb_1(S)<sb_1(S)= b_{1}(S^{[{\bf b}]}).
\end{equation*}
\item As $b_{0}(S)=1$ and $b_{1}(S)=0$, by comparing coefficients of $t^2$ from \eqref{Poincare polynomial of product is product of Poincare polynomials} and Lemma \ref{relations between Betti numbers of S and its Hilbert scheme} (3), we have:
\begin{equation}\label{2nd Betti number incorporating no of 1s}
    b_{2}(S^{[{\bf a}]})= \sum_{i=1}^r b_{2}(S^{[n_{i}]})=\nu_{{\bf a}}(1)b_2(S)+(r-\nu_{{\bf a}}(1)) (b_{2}(S)+1). 
\end{equation}
Therefore, from \eqref{2nd Betti number incorporating no of 1s} and by the hypothesis on $b_2(S)$, 
\begin{equation}\label{eqn: Diff of b2 between to prod hilb}
        b_{2}(S^{[{\bf b}]}) - b_{2}(S^{[{\bf a}]}) = (s-r)(b_2(S)+1)-(\nu_{{\bf b}}(1)-\nu_{{\bf a}}(1)) \neq 0.
    \end{equation}
\end{enumerate}
\end{proof}
\subsection{Classification related to partitions of the same length}
\subsubsection{Surfaces with $b_0(S)>1$}
 First, we recall a couple of definitions from number theory and combinatorics.
\begin{definition}{(cf. \cite[Subsection 2.18]{hardy1952inequalities})}\label{Definition_majorization}
    Let ${\bf a}=(a_1,\ldots,a_r)$ and ${\bf b}=(b_1,\ldots,b_r)$ be two partitions of a positive integer $n$. Then we say ${\bf b}$ \textit{majorizes} ${\bf a}$, denoted by ${\bf b}\succeq {\bf a}$, if
    \begin{equation*}
            \sum_{j=1}^k a_j\leq \sum_{j=1}^k b_j \text{ for all } k \in \{1,2,\ldots,r-1\}. 
    \end{equation*} 

    If ${\bf a}\ne {\bf b}$ and ${\bf b}$ majorizes ${\bf a}$, then we say ${\bf b}$ \textit{ strictly majorizes} ${\bf a}$ and denote it by ${\bf b} \succ {\bf a}$.
\end{definition}
\begin{definition}(cf. \cite[Definition 5.1]{doi:10.1142/S179383090900035X})\label{def: Robin hood}
    Let ${\bf n}=(n_1,n_2,\ldots,n_r)\in \mathbb{N}^r$. A \textit{Robin Hood transformation} is a map $R\colon\mathbb{N}^r\to \mathbb{N}^r$, such that for some $1\leq \ell,j\leq r$ with $n_j>n_\ell$, the vector $R({\bf n})=(N_1,N_2,\ldots,N_r)$ satisfies $N_\ell=n_\ell+1,N_j=n_j-1$, and $N_k=n_k$ for $k\not\in \{\ell,j\}$.
\end{definition}

\begin{lemma}\label{lem: equiv cond for majorization}
    Let ${\bf a}$ and ${\bf b}$ be two partitions of a positive integer $n$. Then  ${\bf b}\succ{\bf a}$ if and only if ${\bf a}$ is obtained from ${\bf b}$ by a finite number of Robin Hood transformations.
\end{lemma}
\begin{proof}
 Follows from \cite[Subsection 2.19, Lemma 2]{hardy1952inequalities} and \cite[Lemma 3.2]{bringmann}.   \end{proof}

\begin{definition}\label{def: log-concavity}
    A sequence of positive integers $\{M_n\}_{n \ge 1}^t$ is said to be \textit{strictly log-concave} if $$M_n^2> M_{n-1}M_{n+1},$$
    for all $2 \le n \le t-1.$
\end{definition}
We recall the following result from  \cite[Lemma 2.3]{bringmann}:

\begin{lemma}\label{lem: equiv condition for log concavity}
    A sequence of positive integers $\{M_n\}_{n \ge 1}^t$ is strictly log-concave if and only if $$M_{m-1}M_{\ell+1}>M_mM_{\ell},$$ for all $2\le \ell+1<m\le t.$
\end{lemma}
We have the following result:
\begin{lemma}\label{lem: log concavity of a n+p/p}
     For any positive integer $p\ge1$, the sequence $\Big\{\frac{n+p}{p}\Big\}_{n\ge 1}^t$ is strictly log-concave.
\end{lemma}
\begin{proof}
  For $2\le n \le t-1$ we have:
  $$\Big(\frac{n+p}{p}\Big)^2-\Big(\frac{n-1+p}{p}\Big)\Big(\frac{n+1+p}{p}\Big)= \frac{1}{p^2}>0.$$
  Hence, the assertion follows.
\end{proof}
\begin{lemma}\label{lem: majorization lemma}
    Let ${\bf a}=(n_1,\ldots,n_r)$ and ${\bf b}=(m_1,\ldots,m_r)$ be two partitions of a positive integer $n$ such that $\bf{b}\succ \bf{a}$. Then for any integer $p\geq1$,
    \begin{enumerate}
        \item $\prod_{i=1}^{r}\dfrac{m_{i}+p}{p}<\prod_{i=1}^{r}\dfrac{n_{i}+p}{p}$\;.
        \item $\prod_{i=1}^{r}\begin{pmatrix} m_{i}+p\\p \end{pmatrix} <\prod_{i=1}^{r}\begin{pmatrix} n_{i}+p\\p \end{pmatrix}$\;.
    \end{enumerate}

\end{lemma}
\begin{proof}
   \begin{enumerate}
       \item  Let us recall the notation from \eqref{eqn: notation for product of (n+p)/p}:
$$\operatorname{Prod}({\bf b}):=\prod_{i=1}^{r}\dfrac{m_{i}+p}{p}.$$
    Following the notation of Definition \ref{def: Robin hood}, there are two possible cases for a Robin Hood transformation $R$. \\ 
    \textbf{Case 1:} If $m_j=m_{\ell}+1,$ then we have:
    \begin{equation*}
        \begin{split}
            R({\bf b}) & =(m_1,\ldots,m_{j-1},m_j-1,m_{j+1},\ldots,m_{\ell-1},m_{\ell}+1,m_{\ell+1},\ldots,m_r) \\ & = (m_1,\ldots,m_{j-1},m_{\ell},m_{j+1},\ldots,m_{\ell-1},m_j,m_{\ell+1},\ldots,m_r).
        \end{split}
    \end{equation*}
    Hence, the multisets consisting of parts of  $R({\bf b})$ and ${\bf b}$ are the same.  Therefore,
    \begin{equation*}
      \operatorname{Prod}(R({\bf b}))=\operatorname{Prod}({\bf b}).  
    \end{equation*}\\
    \textbf{Case 2:} If $m_j> m_{\ell}+1,$ then we have:
    $$R({\bf b}) =(m_1,\ldots,m_{j-1},m_j-1,m_{j+1},\ldots,m_{\ell-1},m_{\ell}+1,m_{\ell+1},\ldots,m_r).$$
    Hence, the multisets consisting of parts of  $R({\bf b})$ and ${\bf b}$ are distinct. Moreover, by Lemma \ref{lem: log concavity of a n+p/p} and Lemma \ref{lem: equiv condition for log concavity}:

  \begin{equation}\label{comparing a and R(a)}
        \begin{split}
             & \operatorname{Prod}(R({\bf b})) \\ = & \operatorname{Prod}((m_1,\ldots,m_{j-1}))\Big( \frac{m_j-1+p}{p} \Big) \operatorname{Prod}((m_{j+1},\ldots,m_{\ell-1}))\Big(\frac{m_{\ell}+1+p}{p}\Big) \\
             &\times \operatorname{Prod}((m_{\ell+1},\ldots,m_{r}))  \\ > & \operatorname{Prod}((m_1,\ldots,m_{j-1}))\Big( \frac{m_j+p}{p} \Big) \operatorname{Prod}((m_{j+1},\ldots,m_{\ell-1}))\Big(\frac{m_{\ell}+p}{p}\Big) \operatorname{Prod}((m_{\ell+1},\ldots,m_{r}))  \\ = &\operatorname{Prod}({\bf b})\;.
        \end{split}
    \end{equation} 

     Since $\bf{b}\succ \bf{a}$, following Lemma \ref{lem: equiv cond for majorization}, there exists a finite number of Robin Hood transformations $R_1,\cdots,R_k$ such that
    $$R_k\circ\cdots\circ R_1({\bf b})= {\bf a}.$$

    We denote 
    \begin{equation*}
        \begin{split}
            R_q\circ\cdots\circ R_1({\bf b}):=(m_1^{(q)},\cdots,m_r^{(q)})={\bf b}^{(q)}\;,
        \end{split}
    \end{equation*}
    for $1\le q \le k$. Set ${\bf b}^{(0)}:=(m_1,\ldots,m_r)$.

Since ${\bf a}\ne {\bf b}$, there exists at least one integer $q$, $1\le q \le k$, such that, 
$$m_j^{(q-1)}>m_{\ell}^{(q-1)}+1.$$
Therefore, following similar calculation as in \eqref{comparing a and R(a)}, we have 
$$\operatorname{Prod}({\bf b}^{(q-1)})<\operatorname{Prod}({\bf b}^{(q)}),$$
and hence,
$$\operatorname{Prod}({\bf b})\le \operatorname{Prod}({\bf b}^{(1)})\le \cdots\le \operatorname{Prod}({\bf b}^{(q-1)})<\operatorname{Prod}({\bf b}^{(q)})\le \cdots\le  \operatorname{Prod}({\bf a}).$$
\item The argument proceeds identically to that of Lemma \ref{lem:refinement lemma} (2). We therefore omit the details.
   \end{enumerate}
\end{proof}
\begin{theorem}\label{thm: same length with b0(S) ge 1}
     Let ${\bf a}=(n_{1}, \dots,n_{r})$ and ${\bf b}=(m_{1}, \dots,m_{r})$ be two partitions of a positive integer $n$ such that ${\bf b} \succ {\bf a}$ .  Then $ S^{[{\bf a}]}$ is not isomorphic to $S^{[{\bf b}]}$, when $b_0(S)>1$.
\end{theorem}
\begin{proof}
As in the Case 1 of Theorem \ref{thm: main theorem for diff lengths}, we have
\begin{equation*}
      b_{0}(S^{[{\bf a}]})= \prod_{i=1}^r b_{0}(S^{[n_{i}]}).  
    \end{equation*}
For $p= b_{0}(S)-1\geq 1$, by Lemma \ref{relations between Betti numbers of S and its Hilbert scheme} (1) and Lemma \ref{lem: majorization lemma} (2), we have:
    \begin{equation*}
            b_{0}(S^{[{\bf b}]})= \prod_{i=1}^{r}\begin{pmatrix} m_{i}+p\\p \end{pmatrix} \\ <\prod_{i=1}^{r}\begin{pmatrix} n_{i}+p\\p \end{pmatrix}= b_{0}(S^{[{\bf a}]}).
    \end{equation*}
    This completes the proof.
\end{proof}

\subsubsection{Surfaces with $b_0(S)=1$}\label{subsection:Surfaces with $b_0(S)=1$}
Since a large class of surfaces $S$ satisfies $b_0(S)=1$ (cf. Table \ref{Table}), including all irreducible surfaces, it is important that we address this case.

\textbf{Surfaces with $b_0(S)=1$ and $b_1(S)=0$:}
The following theorem is analogous to Theorem \ref{thm: main theorem for diff lengths} (3).
\begin{theorem}\label{thm: main theorem for same length for b_0(S)=1 and b_1(S)=0}
     Let ${\bf a}$ and ${\bf b}$ be two distinct partitions of the same length $r$ of any positive integer $n$. Then $S^{[{\bf a}]}$ is not isomorphic to $S^{[{\bf b}]}$,
    whenever $b_0(S)=1$, $b_1(S)=0$ and $\nu_{{\bf a}}(1)\ne \nu_{{\bf b}}(1)$.
\end{theorem}
\begin{proof}
     Putting $r=s$ in the equation \eqref{eqn: Diff of b2 between to prod hilb}, the proof follows.
\end{proof}

\textbf{Surfaces with $b_0(S)=1$ and $b_1(S)\geq 4$:}

Here, we work with surfaces $S$ over $\C$. In \cite[Conjecture 3.1, p.~204]{MR1032930}, it was conjectured (later proved in \cite{GotSoer}) that a generating function for the \textit{$(p,\;q)^{th}$ Hodge numbers} of $S^{[n]}$, denoted by $h^{p,q}(S^{[n]})$, is as follows: 
\begin{equation}\label{eqn:hodge poly for hilb}
    \sum_{n=0}^{\infty} h(S^{[n]} ; x, y) t^n = \prod_{k=1}^{\infty} \left( \frac{\prod_{p+q \text{ odd}} (1 + x^{p+k-1}y^{q+k-1}t^k)^{h^{p,q}(S)}}{\prod_{p+q \text{ even}} (1 - x^{p+k-1}y^{q+k-1}t^k)^{h^{p,q}(S)}} \right)\;,
\end{equation}
where $h(S^{[n]},x,y)=\sum_{p,q}h^{p,q}(S^{[n]})x^py^q$ is the \textit{Hodge polynomial} of $S^{[n]}$. For surfaces with $b_0(S)=h^{0,0}(S)=1$, taking $q=0$ in \eqref{eqn:hodge poly for hilb}, we have the generating series of the Hodge numbers $h^{p,0}(S)$, given as follows (cf. \cite[Proposition 3.3, p.~205]{MR1032930}):
\begin{equation}\label{generating series for h(p,0)}
            \sum_{n\geq 0}\bigg(\sum_{p=0}^{2n}h^{p,0}(S^{[n]})x^p\bigg)t^{n}= \dfrac{(1+xt)^{h^{1,0}(S)}}{(1-t)(1-x^2t)^{h^{2,0}(S)}}\;.
\end{equation}

Therefore, from \eqref{generating series for h(p,0)},
\begin{equation}\label{(p,0)th Hodge number}
       h^{p,0}(S^{[n]})=\text{Coefficients of } x^pt^{n} \text{\;in\;} \tfrac{(1+xt)^{h^{1,0}(S)}}{(1-t)(1-x^2t)^{h^{2,0}(S)}}\;.
\end{equation}
\begin{lemma}\label{Equality of (p,0)th Hodge numbers}
  For $b_0(S)=1$, the following hold:
   \begin{enumerate}
       \item $h^{p,0}(S^{[n]})=h^{p,0}(S^{[m]}),$ for $0\leq p\leq n\leq m$.
       \item $ h^{n+1,0}(S^{[m]})- h^{n+1,0}(S^{[n]})=\begin{pmatrix}
            h^{1,0}(S)\\ n+1
        \end{pmatrix},$ for $1\leq n<m$. \\ Hence, $h^{n+1,0}(S^{[n]})\neq h^{n+1,0}(S^{[m]}),$ whenever $h^{1,0}(S)\geq n+1$. 
   \end{enumerate}

 \end{lemma}
\begin{proof}
\begin{enumerate}
    \item For any $0\leq p\leq n\leq m$, by \eqref{(p,0)th Hodge number}, we have:
   \begin{equation*}
        \begin{split}
            h^{p,0}&(S^{[n]})=\text{ coefficient of }x^pt^n \text{\;in\;}\tfrac{(1+xt)^{h^{1,0}(S)}}{(1-t)(1-x^2t)^{h^{2,0}(S)}}\\
            &=\sum_{t_1+2t_2=p,\;t_i\geq 0}\Big(\text{coefficient of }(xt)^{t_1} \text{ in } (1+xt)^{h^{1,0}(S)}\\&\;\;\;\;\;\times\text{ coefficient of } (x^2t)^{t_2} \text{ in }(1-x^2t)^{-h^{2,0}(S)} \times \text{ coefficient of } (t)^{n-p+t_2} \text{ in }(1-t)^{-1}\Big)\\
            &=\sum_{t_1+2t_2=p,\;t_i\geq 0}\Big(\text{coefficient of }(xt)^{t_1} \text{ in } (1+xt)^{h^{1,0}(S)}\\&\;\;\;\;\;\times\text{ coefficient of } (x^2t)^{t_2} \text{ in }(1-x^2t)^{-h^{2,0}(S)}\times 1\Big)\\
            &=\sum_{t_1+2t_2=p,\;t_i\geq 0}\Big(\text{coefficient of }(xt)^{t_1} \text{ in } (1+xt)^{h^{1,0}(S)}\\&\;\;\;\;\;\times\text{ coefficient of } (x^2t)^{t_2} \text{ in }(1-x^2t)^{-h^{2,0}(S)} \times \text{ coefficient of } (t)^{m-p+t_2} \text{ in }(1-t)^{-1}\Big)
            \\&= \text{ coefficient of }x^pt^m \text{\;in\;}\tfrac{(1+xt)^{h^{1,0}(S)}}{(1-t)(1-x^2t)^{h^{2,0}(S)}}=h^{p,0}(S^{[m]}).
        \end{split}
    \end{equation*}
    \item Consider the set $U$ defined as follows:
\begin{equation*}
U:=\big\{(t_1,t_2,t_3)\in \mathbb{Z}_{\geq 0}^3\mid \sum_{i=1}^3t_i=n,t_1+2t_2=n+1\big\}\;.    
\end{equation*}
 $U$ can be written as the set $V$ given as follows:
\begin{equation*}
V:=\big\{(t_1,t_2,t_3)\in \mathbb{Z}_{\geq 0}^3\mid \sum_{i=1}^3t_i=n,t_1+2t_2=n+1,t_2>0\big\}\;.    
\end{equation*}
Therefore, for $1\leq n$, from \eqref{(p,0)th Hodge number}, we have:
\begin{equation*}
        \begin{split}
            h^{n+1,0}&(S^{[n]})=\text{ coefficient of }x^{n+1}t^n \text{\;in\;}\tfrac{(1+xt)^{h^{1,0}(S)}}{(1-t)(1-x^2t)^{h^{2,0}(S)}}\\ &=\sum_{(t_1,t_2,t_3)\in U}\Big(\text{coefficient of }(xt)^{t_1} \text{ in } (1+xt)^{h^{1,0}(S)}\\&\;\;\;\;\;\times\text{coefficient of } (x^2t)^{t_2} \text{ in }(1-x^2t)^{-h^{2,0}(S)}\times \text{coefficient of } t^{t_3} \text{ in }(1-t)^{-1}\Big)\\
            &=\sum_{(t_1,t_2,t_3)\in V} \Big(\text{coefficient of }(xt)^{t_1} \text{ in } (1+xt)^{h^{1,0}(S)}\\&\;\;\;\;\;\times\text{coefficient of } (x^2t)^{t_2} \text{ in }(1-x^2t)^{-h^{2,0}(S)}\times \text{coefficient of } t^{t_3} \text{ in }(1-t)^{-1}\Big)\\
            &=\sum_{\substack{t_1+2t_2=n+1\\t_1\geq 0,\;t_2>0}}\Big(\text{coefficient of }(xt)^{t_1} \text{ in } (1+xt)^{h^{1,0}(S)}\\&\;\;\;\;\;\times\text{coefficient of } (x^2t)^{t_2} \text{ in }(1-x^2t)^{-h^{2,0}(S)}\Big).
        \end{split}
    \end{equation*}
    For $n< m$, from \eqref{(p,0)th Hodge number}, we have:
\begin{equation*}
        \begin{split}
            h^{n+1,0}&(S^{[m]})=\text{coefficient of }x^{n+1}t^m \text{\;in\;}\tfrac{(1+xt)^{h^{1,0}(S)}}{(1-t)(1-x^2t)^{h^{2,0}(S)}}\\&=\sum_{\substack{t_1+2t_2=n+1\\t_1+t_2+t_3=m,\;t_i\geq 0}}\Big(\text{coefficient of }(xt)^{t_1} \text{ in } (1+xt)^{h^{1,0}(S)}\\&\;\;\;\;\;\times\text{coefficient of } (x^2t)^{t_2} \text{ in }(1-x^2t)^{-h^{2,0}(S)}\times \text{coefficient of } t^{t_3} \text{ in }(1-t)^{-1}\Big)\\
        &=\sum_{\substack{t_1+2t_2=n+1\\t_1+t_2+t_3=m,\;t_1\geq 0,\;t_3\geq 0,\; t_2>0}}\Big(\text{coefficient of }(xt)^{t_1} \text{ in } (1+xt)^{h^{1,0}(S)}\\&\;\;\;\;\;\times\text{coefficient of } (x^2t)^{t_2} \text{ in }(1-x^2t)^{-h^{2,0}(S)}\times \text{coefficient of } t^{t_3} \text{ in }(1-t)^{-1}\Big)\\&+\Big(\text{coefficient of }(xt)^{n+1} \text{ in } (1+xt)^{h^{1,0}(S)}\times\text{coefficient of } 1 \text{ in }(1-x^2t)^{-h^{2,0}(S)}\\&\;\;\;\;\;\times \text{coefficient of } {t}^{m-(n+1)} \text{ in }(1-t)^{-1}\Big)\\
            &=\sum_{\substack{t_1+2t_2=n+1\\t_1\geq 0,\;t_2>0}}\Big(\text{coefficient of }(xt)^{t_1} \text{ in } (1+xt)^{h^{1,0}(S)}\\&\;\;\;\;\;\times\text{coefficient of } (x^2t)^{t_2} \text{ in }(1-x^2t)^{-h^{2,0}(S)}\times 1\Big)\\&\;\;\;\;\;+\Big(\text{coefficient of }(xt)^{n+1} \text{ in } (1+xt)^{h^{1,0}(S)}\times 1 \times 1\Big).
        \end{split}
    \end{equation*}
   Therefore, we have:
      \begin{equation*}\label{quantative difference between h(p,0)}
          \begin{split}
              h^{n+1,0}(S^{[m]})-h^{n+1,0}(S^{[n]})&=\text{ coefficient of }x^{n+1}t^m \text{\;in\;}\tfrac{(1+xt)^{h^{1,0}(S)}}{(1-t)(1-x^2t)^{h^{2,0}(S)}}\\& -\text{ coefficient of }x^{n+1}t^n \text{\;in\;}\tfrac{(1+xt)^{h^{1,0}(S)}}{(1-t)(1-x^2t)^{h^{2,0}(S)}}\\&=\text{ coefficient of }(xt)^{n+1} \text{ in }(1+xt)^{h^{1,0}(S)}\\&=\begin{pmatrix}
            h^{1,0}(S)\\ n+1
        \end{pmatrix}.
          \end{split}
      \end{equation*}
Hence, the assertion follows.
\end{enumerate}
\end{proof}

By the K\"unneth formula for $(p,0)^{\text{th}}$ Hodge numbers (cf. \cite[p.~105]{GH2014}), we have: 
\begin{equation}\label{Kunneth formula for h(p,0)}
    h^{p,0}(S^{[{\bf a}]})=\sum_{\substack{t_1+\cdots+t_r=p\\t_i\geq 0,\; 1\leq i \leq r}}\bigg(\prod_{i=1}^rh^{t_i,0}(S^{[n_i]})\bigg),
\end{equation}
where ${\bf a}=(n_{1}, n_{2}, \dots,n_{r})$.
\begin{theorem}\label{thm:classification_same length_condition involves first betti number and minimal part}
     Let ${\bf a}=(n_{1}, \dots,n_{r})$ and ${\bf b}=(m_{1}, \dots,m_{r})$ be two distinct partitions of any positive integer $n$ and $e = \operatorname{max}\{i \in \{1,2,\ldots,r\} : n_i \ne m_i\}$.  Then $S^{[{\bf a}]}$ is not isomorphic to $S^{[{\bf b}]}$, whenever  $b_0(S)=1$ and $b_1(S)\ge 2(\min\{n_e,m_e\}+1)$.
\end{theorem}
\begin{proof}
For notational convenience, we consider the parts of the partition in non-decreasing order. Equivalently, we show that $S^{[{\bf a}]}$ is not isomorphic to $S^{[{\bf b}]}$ whenever $b_0(S)=1$ and $b_1(S)\ge 2(\min\{n_j,m_j\}+1)$, where  $j = \operatorname{min}\{i \in \{1,2,\ldots,r\} : n_i \ne m_i\}$.

As $n_j\neq m_j$, without loss of generality, we assume $n_j<m_j$.  Moreover, suppose $n_j$ occurs $k$ times in $\{n_j,n_{j+1},\dots , n_{r}\}$.  Then, the parts of the partitions $\bf{a}$ and $\bf{b}$ can be compared as follows:

$$\renewcommand{\arraystretch}{1.3}
\begin{array}[c]{ccccccccccccccccc}
n_1 & \leq &\cdots & \leq & n_{j-1}& \leq & n_j & =&\cdots & = & n_{j+k-1}&< & n_{j+k} & \leq & \cdots & \leq & n_r\\
\big\|&  & & & \big\| &  & \bigwedge &&& & & &&  & & & \\
m_1 & \leq &\cdots & \leq & m_{j-1}& \leq & m_j &\leq&\cdots & \leq & m_{j+k-1}&\leq & m_{j+k} & \leq  & \cdots & \leq & m_r\\
\multicolumn{5}{c}{\underbrace{\hspace{3.2cm}}_{1\; \le \;i \;\le\; j-1}}
& \multicolumn{6}{c}{\underbrace{\hspace{4.2cm}}_{j\; \le \;i \;\le\; j+k-1}}
& \multicolumn{6}{c}{\underbrace{\hspace{4.2cm}}_{j+k\; \le\; i \;\le\; r}}
\end{array}$$
That is, we have the following equalities and inequalities:

\begin{equation}\label{ordering of the parts}
        n_i=m_i \text{ for all }1 \leq i \leq j-1\text{ and }n_j+1 \leq m_u\text{ for all }j \leq u \leq r.
    \end{equation}
\begin{equation}\label{repetition of n_j}
        n_{j+l}=n_j \text{ for all }0 \leq l \leq k-1\text{ and }n_j+1\leq n_{v} \text{ for all }j+k \leq v \leq r.
    \end{equation} 
Let $T$ be the following set:
\begin{equation}\label{eqn_definition of set T}
    T:=\{(t_1,\ldots,t_r)\in \mathbb{Z}_{\geq 0}^r\mid \sum_{i=1}^rt_i=n_j+1\}.
\end{equation}
Let $T^{\prime}$ and $T^{\prime \prime}$ be the following subsets of $T$.
\begin{equation*}
\begin{split}
    T^{\prime}&:=\{(t_1,\ldots,t_r)\in T\mid t_i\leq n_j \text{\;for\;all\;}1\leq i\leq r\},\\
    T^{\prime \prime}&:=\{(t_1,\ldots,t_r)\in T\mid t_i= n_j+1 \text{\;for\;some\;}1\leq i\leq r\}.
    \end{split}
\end{equation*}
Clearly, we have the following decomposition of $T$, as defined in \eqref{eqn_definition of set T}:
\begin{equation}\label{eqn_decomposition of set T}
T=T^{\prime}\cup T^{\prime \prime} \text{\;and\;}T^{\prime}\cap T^{\prime \prime}=\emptyset.  
\end{equation}
We now consider the following subsets of $T^{\prime \prime}$ itself.
\begin{equation*}
\begin{split}
    T^{\prime \prime}_1&:=\{(t_1,\ldots,t_r)\in T^{\prime \prime}\mid t_i= n_j+1 \text{\;for\;some\;}1\leq i\leq j-1\},\\
    T^{\prime \prime}_2&:=\{(t_1,\ldots,t_r)\in T^{\prime \prime}\mid t_i= n_j+1 \text{\;for\;some\;}j\leq i\leq j+k-1\},\\
    T^{\prime \prime}_3&:=\{(t_1,\ldots,t_r)\in T^{\prime \prime}\mid t_i= n_j+1 \text{\;for\;some\;}j+k\leq i\leq r\}.
 \end{split}   
\end{equation*}
Then, clearly, we have:
\begin{equation}\label{eqn_decomposition of T prime prime}
 T^{\prime \prime}=T^{\prime\prime}_1\cup T^{\prime \prime}_2\cup T^{\prime \prime}_3  \text{\;and\;}T^{\prime \prime}_{i_1}\cap T^{\prime \prime}_{i_2}=\emptyset \text{\;for\;}1\leq i_1\neq i_2 \leq 3.   
\end{equation}
 $\textbf{Claim:\;\;\;\;\;\;\;\;\;\;} h^{n_j+1,0}(S^{[{\bf a}]})\neq h^{n_j+1,0}(S^{[{\bf b}]}), \text{\;whenever\;} h^{1,0}(S)\geq n_j+1.$\\
\textit{Proof of the claim:}
By \eqref{(p,0)th Hodge number} and \eqref{Kunneth formula for h(p,0)}, we have:
\begin{equation}\label{(n_j+1,0)th Hodge number for a}
        \begin{split}
            h^{n_j+1,0}(S^{[{\bf a}]})& =\sum_{(t_1,\ldots,t_r)\in T}\bigg(\prod_{i=1}^{r}\text{ coefficient of }x^{t_i}t^{n_i} \text{ in } \tfrac{(1+xt)^{h^{1,0}(S)}}{(1-t)(1-x^2t)^{h^{2,0}(S)}}\bigg)\;.
        \end{split}
    \end{equation}
    Similarly, we have: 
\begin{equation}\label{(n_j+1,0)th Hodge number for b}
        \begin{split}
            h^{n_j+1,0}(S^{[{\bf b}]})& =\sum_{(t_1,\ldots,t_r)\in T}\bigg(\prod_{i=1}^{r}\text{ coefficient of }x^{t_i}t^{m_i} \text{ in } \tfrac{(1+xt)^{h^{1,0}(S)}}{(1-t)(1-x^2t)^{h^{2,0}(S)}}\bigg)\;.
        \end{split}
    \end{equation}
 From \eqref{eqn_decomposition of set T} and \eqref{eqn_decomposition of T prime prime}, we have the following decomposition of $T$, as defined in \eqref{eqn_definition of set T}:
 \begin{equation*}
   T=T^{\prime}\cup T^{\prime \prime}_1\cup T^{\prime \prime}_2\cup T^{\prime \prime}_3.  
 \end{equation*}
Therefore, by \eqref{(n_j+1,0)th Hodge number for a} and \eqref{(n_j+1,0)th Hodge number for b}, to compare the Hodge numbers $h^{n_j+1,0}(S^{[{\bf a}]})$ and $h^{n_j+1,0}(S^{[{\bf b}]})$, it is enough to consider the following mutually exclusive and exhaustive cases.\\
\textbf{Case 1:} $(t_1,\ldots,t_r)\in  T^{\prime}$\\
For any $(t_1,\ldots,t_r)\in  T^{\prime}$, partition ${\bf a}=(n_1,\ldots,n_r)$ and any $1\leq h \leq l \leq r$, we denote:
\begin{equation*}
 \Prod_h^{l}({\bf a}, (t_1,\ldots, t_r)):=\prod_{i=h}^{l}\text{ coefficient of }x^{t_i}t^{n_i} \text{ in } \tfrac{(1+xt)^{h^{1,0}(S)}}{(1-t)(1-x^2t)^{h^{2,0}(S)}}\;.   
\end{equation*}
As $n_i=m_i$ for all $1\leq i \leq j-1$, we have 
\begin{equation}\label{T prime_first range}
 \Prod_1^{j-1}({\bf a}, (t_1,\ldots, t_r))=\Prod_1^{j-1}({\bf b}, (t_1,\ldots, t_r)).   
\end{equation}
By \eqref{ordering of the parts} and \eqref{repetition of n_j} , we have $t_i\leq n_j =n_i <m_i$ for all $j\leq i \leq j+k-1$.  Therefore, by Lemma \ref{Equality of (p,0)th Hodge numbers} (1), 
\begin{equation}\label{T prime_second range}
 \Prod_{j}^{j+k-1}({\bf a}, (t_1,\ldots, t_r))=\Prod_{j}^{j+k-1}({\bf b}, (t_1,\ldots, t_r)). \end{equation}
 By \eqref{ordering of the parts} and \eqref{repetition of n_j} , we have either $t_i \leq n_j < n_j+1 \leq n_i \leq m_i$ or $t_i \leq n_j < n_j+1 \leq m_i \leq n_i$, for all $j+k\leq i \leq r$.  Therefore, by Lemma \ref{Equality of (p,0)th Hodge numbers} (1),
 \begin{equation}\label{T prime_third range}
 \Prod_{j+k}^{r}({\bf a}, (t_1,\ldots, t_r))=\Prod_{j+k}^{r}({\bf b}, (t_1,\ldots, t_r)). \end{equation}
So, for any $(t_1,\ldots,t_r)\in  T^{\prime}$, by \eqref{T prime_first range}, \eqref{T prime_second range} and \eqref{T prime_third range}, R.H.S of \eqref{(n_j+1,0)th Hodge number for a} and R.H.S of \eqref{(n_j+1,0)th Hodge number for b} are equal.\\ 
\textbf{Case 2:} $(t_1,\ldots,t_r)\in  T^{\prime \prime}_1 $\\
 The corresponding summand of the R.H.S of \eqref{(n_j+1,0)th Hodge number for a} and that of R.H.S of \eqref{(n_j+1,0)th Hodge number for b} are equal as $n_i=m_i$ for all $1 \leq i \leq j-1$.\\
\textbf{Case 3:} $(t_1,\ldots,t_r)\in T^{\prime \prime }_3 $\\
By \eqref{ordering of the parts} and \eqref{repetition of n_j} , we have either $t_i \leq n_j+1 \leq  n_i \leq m_i$ or $t_i \leq n_j+1 \leq m_i \leq n_i$, where $(t_1,\ldots,t_r)\in T^{\prime\prime}_3$.  Therefore, the corresponding summand of the R.H.S of \eqref{(n_j+1,0)th Hodge number for a} and that of R.H.S of \eqref{(n_j+1,0)th Hodge number for b} are equal by Lemma \ref{Equality of (p,0)th Hodge numbers} (1).\\
\textbf{Case 4:} $(t_1,\ldots,t_r)\in T^{\prime \prime }_2 $\\
We have $t_{i_0}=n_j+1$ for some $j\leq i_0 \leq j+k-1$ and $t_{i}=0$ for all $j\leq i(\neq i_0) \leq j+k-1$.  Therefore, for this case, we have two mutually exclusive and exhaustive subcases.\\
\textbf{(a)} $0=t_i<n_i< m_i$ for all $j\leq i(\neq i_0) \leq j+k-1$: The corresponding product terms of the R.H.S of \eqref{(n_j+1,0)th Hodge number for a} and that of R.H.S of \eqref{(n_j+1,0)th Hodge number for b} are equal by Lemma \ref{Equality of (p,0)th Hodge numbers} (1).\\
\textbf{(b)} $t_{i_0}=n_j+1$: The corresponding product term of the R.H.S of \eqref{(n_j+1,0)th Hodge number for a} and that of R.H.S of \eqref{(n_j+1,0)th Hodge number for b} are unequal by Lemma \ref{Equality of (p,0)th Hodge numbers} (2), whenever $h^{1,0}(S)\geq n_j+1$.\\\\
Therefore, from \eqref{(n_j+1,0)th Hodge number for a} and \eqref{(n_j+1,0)th Hodge number for b}, Lemma \ref{Equality of (p,0)th Hodge numbers} (2) and the above case wise discussion, we have:
\begin{equation*}
        \begin{split}
            h^{n_j+1,0}(S^{[{\bf b}]})&-h^{n_j+1,0}(S^{[{\bf a}]})\\
            &=\sum_{(t_1,\ldots,t_r)\in T}\bigg(\prod_{i=1}^{r}\text{ coefficient of }x^{t_i}t^{m_i} \text{ in } \tfrac{(1+xt)^{h^{1,0}(S)}}{(1-t)(1-x^2t)^{h^{2,0}(S)}}\bigg)\\&-\sum_{(t_1,\ldots,t_r)\in T}\bigg(\prod_{i=1}^{r}\text{ coefficient of }x^{t_i}t^{n_i} \text{ in } \tfrac{(1+xt)^{h^{1,0}(S)}}{(1-t)(1-x^2t)^{h^{2,0}(S)}}\bigg)\\
        &=\sum_{(t_1,\ldots,t_r)\in T^{\prime \prime }_2}\bigg(\prod_{i=1}^{r}\text{ coefficient of }x^{t_i}t^{m_i} \text{ in } \tfrac{(1+xt)^{h^{1,0}(S)}}{(1-t)(1-x^2t)^{h^{2,0}(S)}}\bigg)\\&-\sum_{(t_1,\ldots,t_r)\in T^{\prime \prime }_2}\bigg(\prod_{i=1}^{r}\text{ coefficient of }x^{t_i}t^{n_i} \text{ in } \tfrac{(1+xt)^{h^{1,0}(S)}}{(1-t)(1-x^2t)^{h^{2,0}(S)}}\bigg)\\&=\card( T^{\prime \prime}_2) \begin{pmatrix}
            h^{1,0}(S)\\ n_j+1\end{pmatrix} =k \begin{pmatrix}
            h^{1,0}(S)\\ n_j+1\end{pmatrix}.
        \end{split}
    \end{equation*}
    Therefore, as $k\ge 1$ (cf. \eqref{repetition of n_j}), $h^{n_j+1,0}(S^{[{\bf b}]})\ne h^{n_j+1,0}(S^{[{\bf a}]})$, whenever $h^{1,0}(S)\geq n_j+1$.  Hence, the claim follows.
    
    Moreover, we have, $h^{1,0}(S)+h^{0,1}(S)=b_1(S)$ and $h^{1,0}(S)=h^{0,1}(S)$ (cf. \cite[p.~105-106]{GH2014}). So, the assertion follows.
\end{proof}

\subsubsection{Surfaces with $\chi(S) \ge 3$}
The generating function of the \textit{Euler characteristic} $\chi(S^{[n]})$ of $S^{[n]}$ is given as follows (cf. \cite[Equation $(2)$, Theorem 0.1, p.~193]{MR1032930}): 
\begin{equation}\label{26}
    \sum_{n=0}^{\infty}\chi(S^{[n]})q^n=\prod_{n=1}^{\infty}\frac{1}{(1-q^n)^{\chi(S)}}\;.
\end{equation}

By a \textit{$k$-coloured partition of $n$}, we mean a partition of $n$ where each part is assigned one of the given $k$ colours. The total number of $k$-coloured partitions is denoted by $p_k(n)$. Observe that for $k=1$, we recover the usual partition function $p(n)$. The generating series of $p_k(n)$ is given as follows (cf. \cite[First equation, p.~616]{MR4287510}):  
\begin{equation}\label{27}
    \sum_{n=0}^{\infty}p_k(n)q^n=\prod_{n=1}^{\infty}\frac{1}{(1-q^n)^k}\;.
\end{equation}
Hence, following \eqref{26} and \eqref{27}, we find that:
\begin{equation}\label{28}
    \chi(S^{[n]})=p_{\chi(S)}(n).
\end{equation} 

\begin{definition}\label{Defn_colour partition respects majorising order}
 Given ${\bf a}=(n_1,n_2,\ldots,n_r)\in \mathbb{N}^r$, we define 
 \begin{equation*}
     p_k({\bf a}):=\prod_{i=1}^r p_k(n_i),
 \end{equation*}
 where $p_k(n_i)$ is the number of $k$-coloured partitions of $n_i$, $1\leq i \leq r$.
\end{definition}
We now recall a strict inequality involving $p_k({\bf a})$ (cf. \cite[Theorem 1.4]{bringmann}).
\begin{lemma}\label{majorization inequality}
    Let ${\bf a}=(n_1,\ldots,n_r)$ and ${\bf b}=(m_1,\ldots,m_r)$ be two partitions of a positive integer $n$, of same length, such that ${\bf b}\succ {\bf a}$. Then  $p_k({\bf a}) > p_k({\bf b})$, for $k\geq 3$.  
\end{lemma}
Using Lemma \ref{majorization inequality}, we obtain the following.
\begin{theorem}\label{Classification_same length case_using majorization}
Let ${\bf a}=(n_1,\ldots,n_r)$ and ${\bf b}=(m_1,\ldots,m_r)$ be two partitions of a positive integer $n$, of same length, such that ${\bf b}\succ {\bf a}$.  Then $S^{[{\bf a}]}$ is not isomorphic to $S^{[{\bf b}]}$, whenever $\chi(S)\ge 3$.
\end{theorem}
\begin{proof}
    Following \eqref{28}, the K\"unneth formula and  Lemma \ref{majorization inequality}, we have: 
    \begin{equation*}
        \begin{split}
            \chi(S^{[{\bf a}]})=\prod_{i=1}^r\chi(S^{[n_i]}) =\prod_{i=1}^r p_{\chi(S)}(n_i)=p_{\chi(S)}({\bf a})> p_{\chi(S)}({\bf b})=\chi(S^{[{\bf b}]}).
        \end{split}
    \end{equation*}
Hence, the assertion follows. 
\end{proof}

  As there are several examples of surfaces $S$ with $b_0(S)=1$ and $b_1(S)=0$ (cf. Table \ref{Table}), we record the following consequence of Theorem \ref{Classification_same length case_using majorization}.

\begin{corollary}\label{corollary_classification using majorization_same length case}
     Let ${\bf a}$ and ${\bf b}$ be two partitions of a given positive integer $n$, of the same length, such that ${\bf b}\succ {\bf a}$.  Then $ S^{[{\bf a}]}$ is not isomorphic to $S^{[{\bf b}]}$ when $b_0(S)=1$ and $b_1(S)=0$.
\end{corollary}
\begin{proof}
    Since $b_2(S)\ge 1$ for any smooth projective surface $S$, the result follows directly from Theorem \ref{Classification_same length case_using majorization}.
\end{proof}
\begin{remark}\label{Remark_validity over alg closed fields over finite fields}
Theorem \ref{thm: main theorem for diff lengths}, Theorem \ref{thm: same length with b0(S) ge 1}, Theorem \ref{thm: main theorem for same length for b_0(S)=1 and b_1(S)=0}, and Theorem \ref{Classification_same length case_using majorization} hold for any smooth projective surface over $\overline{\F}_q$ as well, where $\F_q$ is the finite field with $q$ elements, as so do the generating function for $b_i(S^{[n]})$ and $\chi(S^{[n]})$ (cf. \cite[Theorem 0.1, Equation 1(b) and Equation (2) p.~193]{MR1032930}).
\end{remark}

 \section*{Declaration of competing interest}
 There is no conflict of interest.

\section*{Acknowledgements}
   The first-named author would like to thank the Indian Institute of Technology Madras for financial support (Office order No.F.ARU/R10/IPDF/2024). The second-named author received financial support from the Prime Minister's Research Fellowship (PMRF, ID: 2503482). The authors would like to thank A. Dey, C. Gangopadhyay, L. G\"ottsche, D. S. Nagaraj and C. Voisin for their comments and suggestions about this article.

\end{document}